\documentclass[12pt,reqno]{amsart}
\usepackage{amssymb}
\usepackage{amsxtra}
\usepackage{mathrsfs}
\usepackage[all]{xy}
\usepackage{cite}
\usepackage{paralist}
\usepackage[hypertex]{hyperref}
\usepackage[active]{srcltx} 
\newtheorem{theorem}{Theorem}[section]
\newtheorem{lemma}[theorem]{Lemma}
\newtheorem{prop}[theorem]{Proposition}
\newtheorem{corollary}[theorem]{Corollary}
\newtheorem{question}[theorem]{Question}
\theoremstyle{definition}
\newtheorem{definition}[theorem]{Definition}
\theoremstyle{remark}
\newtheorem{example}[theorem]{Example}
\newtheorem{remark}[theorem]{Remark}
\newtheorem*{ackn}{Acknowledgments}
\DeclareMathOperator{\Hom}{Hom}
\DeclareMathOperator{\Ker}{Ker}
\DeclareMathOperator{\Tor}{Tor}
\DeclareMathOperator{\Der}{Der}
\newcommand*{\dL}{\mathrm{L}}
\newcommand*{\lmod}{\mbox{-}\!\mathop{\mathsf{mod}}}
\newcommand*{\lbarmod}{\mbox{-}{\mathop{\underline{\mathsf{mod}}}}}
\newcommand*{\rmod}{\mathop{\mathsf{mod}}\!\mbox{-}}

\newcommand*{\bimod}{\mbox{-}\!\mathop{\mathsf{mod}}\!\mbox{-}}

\newcommand*{\Fr}{\mathsf{Fr}}
\newcommand*{\barFr}{\underline{\mathsf{Fr}}}
\newcommand*{\Vect}{\mathsf{Vect}}
\newcommand*{\Ptens}{\mathop{\widehat\otimes}}
\newcommand*{\ptens}[1]{\mathop{\widehat\otimes}_{#1}}
\newcommand*{\Lptens}[1]{\mathop{\widehat\otimes}_{#1}\nolimits^{\dL}}
\newcommand*{\tens}[1]{\mathop{\otimes}_{#1}}
\newcommand*{\id}{1}
\newcommand*{\wt}{\widetilde}
\newcommand*{\h}{\mathbf h}
\newcommand*{\CC}{\mathbb C}
\newcommand*{\Z}{\mathbb Z}
\newcommand*{\R}{\mathbb R}
\newcommand*{\DD}{\mathbb D}
\newcommand*{\cO}{\mathscr O}
\newcommand*{\cF}{\mathscr F}
\newcommand*{\cG}{\mathscr G}
\newcommand*{\cI}{\mathscr I}
\newcommand*{\cH}{\mathscr H}
\newcommand*{\cC}{\mathscr C}
\newcommand*{\fm}{\mathfrak m}

\newcommand*{\cP}{\mathscr P}
\newcommand*{\bD}{\mathsf D}

\newcommand*{\eps}{\varepsilon}
\newenvironment{mycompactenum}{\pltopsep=5pt\begin{compactenum}[\upshape (i)]}%
{\end{compactenum}}
\newcommand*{\lar}{\leftarrow}
\newcommand*{\xla}{\xleftarrow}
\newcommand*{\xra}{\xrightarrow}
\newcommand{\lriso}{\stackrel{\textstyle\sim}{\smash\longrightarrow%
\vphantom{\scriptscriptstyle{_1}}}}
\numberwithin{equation}{section}
\begin{document}
\title{Open embeddings and pseudoflat epimorphisms}

\subjclass[2010]{46M18, 46H25, 46E25, 32A38, 16E30}
\author{O. Yu. Aristov}
\address{Oleg Yu. Aristov}
\email{aristovoyu@inbox.ru}
\author{A. Yu. Pirkovskii}
\address{Alexei Yu. Pirkovskii, Faculty of Mathematics,
HSE University,
6 Usacheva, 119048 Moscow, Russia}
\email{aupirkovskii@hse.ru}
\thanks{This work was supported by the RFBR grant no. 19-01-00447.}
\dedicatory{Dedicated to Professor Alexander Ya. Helemskii on the occasion of his 75th birthday}
\date{}

\begin{abstract}
We characterize open embeddings of Stein spaces and of $C^\infty$-manifolds
in terms of certain flatness-type conditions on the respective homomorphisms
of function algebras.
\end{abstract}

\maketitle

\section{Introduction}
\label{sect:intro}

Our main motivation comes from the following fact in algebraic geometry.
If $(X,\cO_X)$ and $(Y,\cO_Y)$ are affine schemes, then a morphism
$f\colon (Y,\cO_Y)\to (X,\cO_X)$ is an open
embedding if and only if the respective homomorphism $f^\bullet\colon \cO(X)\to\cO(Y)$
is a flat epimorphism of finite presentation \cite[17.9.1]{EGA4}.
We are interested in complex analytic and smooth versions of this result.
Specifically, given a morphism $f\colon (Y,\cO_Y)\to (X,\cO_X)$ of Stein spaces,
we are looking for a condition on $f^\bullet\colon \cO(X)\to\cO(Y)$ that is necessary and
sufficient for $f$ to be an open embedding. A similar question makes sense for
$C^\infty$-manifolds. To get a reasonable answer, we equip the algebras of holomorphic
and smooth functions with their canonical Fr\'echet space topologies and consider them
as functional analytic objects \cite{X1,X2}.

It is easy to see that the above-mentioned algebraic result does not extend {\em verbatim} to
the complex analytic case. Indeed, if $U$ is an open subset of a Stein space $(X,\cO_X)$, then
$\cO(U)$ is normally not flat as a Fr\'echet $\cO(X)$-module. This observation is essentially
due to M.~Putinar \cite{Put_SSh1} (see also \cite{EschmPut})
and is closely related to the spectral theory of linear operators on Banach spaces.
Actually, if $\cO(U)$ were flat over $\cO(\CC)$ for every open subset $U\subset\CC$, then each
Banach space operator would possess Bishop's property $(\beta)$, which is not the case \cite{Bishop}.
For a direct proof of the fact that $\cO(\DD)$ is not flat over $\cO(\CC)$
(where $\DD\subset\CC$ is the open unit disc), see \cite{Pir_Stein}.

A reasonable substitute for the flatness property was introduced by J.~L.~Taylor \cite{T2}.
Given a continuous homomorphism $\varphi\colon A\to B$ of
Fr\'echet algebras,
he says that $\varphi$ is a localization if, for each Fr\'echet $B$-bimodule $M$,
the induced map of the continuous Hochschild homology $\cH_\bullet(A,M)\to\cH_\bullet(B,M)$
is an isomorphism. Taylor also proved that, if $A$ and $B$ are nuclear, then the above
condition means precisely that
(i) $\Tor_i^A(B,B)=0$ for all $i\ge 1$,
and (ii) $\Tor_0^A(B,B)\cong B$ canonically.
Homomorphisms satisfying (i) and (ii) were rediscovered several times under different names
\cite{Dicks,GL,NR,Meyer,BBK}, both in the purely algebraic and in the functional analytic
contexts (see Remark~\ref{rem:hist} for historical details).
We adopt the terminology of \cite{GL} and call such maps homological
epimorphisms. To be more precise, there are two types of homological epimorphisms
in the functional analytic setting, weak and strong homological epimorphisms.
For nuclear Fr\'echet algebras, weak homological epimorphisms are the same as
Taylor's localizations, while strong homological epimorphisms are the same as Taylor's absolute
localizations.
See Section \ref{sect:psd} for details.

The fundamental (and chronologically the first) example of a weak homological epimorphism
that is not necessarily flat is the restriction map
$\cO(\CC^n)\to\cO(U)$, where $U$ is a Stein open subset of $\CC^n$
(i.e., a domain of holomorphy).
This fact was proved by Taylor \cite[Prop. 4.3]{T2} and was the main motivation for him
to introduce weak homological epimorphisms.
The second author \cite[Theorem 3.1]{Pir_Stein} observed that the same result
holds if we replace $\CC^n$ by an arbitrary Stein manifold.
Recently, F.~Bambozzi, O.~Ben-Bassat and K.~Kremnizer \cite{BBBK}, working in
the setting of bornological algebras, proved
that the above property actually characterizes open embeddings of Stein spaces
(not only over $\CC$).

Other examples of homological epimorphisms in the functional analytic context
can be found in \cite{T2,T3,Dos_faappl,Dos_Kth,Dos_colloq,Pir_stbflat,Pir_qfree}.

In the present paper, we introduce a wider class of Fr\'echet algebra homomorphisms
$A\to B$ that we call {\em $n$-pseudoflat epimorphisms}
(where $n$ is a fixed nonnegative integer). Such homomorphisms
are defined by the conditions that
$\Tor_i^A(B,B)=0$ for all $1\le i\le n$
and $\Tor_0^A(B,B)\cong B$ canonically. For $n=1$, pseudoflat epimorphisms
were introduced by G.~M.~Bergman and W.~Dicks \cite{BerDic} in the purely algebraic
setting. They also appear naturally in \cite{Schofield,AH,Ba_Posic},
for example. As far as we know, pseudoflat epimorphisms
were not considered before in the functional analytic framework.
Our main results are Theorems~\ref{holeqcond} and~\ref{BIGTHsmman},
which characterize open embeddings of Stein spaces and of smooth manifolds in
terms of pseudoflat epimorphisms.

The paper is organized as follows. Section~\ref{sect:prelim} contains some
preliminaries from homological algebra in categories of Fr\'echet modules.
Our main reference is \cite{X1}; some facts that are missing in \cite{X1} can
be found in \cite{T1,EschmPut,Pir_qfree}. In Section~\ref{sect:psd},
we introduce $n$-pseudoflat epimorphisms of Fr\'echet algebras, give some
examples, and characterize epimorphisms, $0$-pseudoflat epimorphisms,
and $1$-pseudoflat epimorphisms in terms of noncommutative differential forms.
In particular, we show that not every Fr\'echet algebra
epimorphism is $0$-pseudoflat (in contrast to the purely algebraic case).
Our main results are contained in Sections~\ref{sect:Stein} and~\ref{sect:C-inf}.
In Section~\ref{sect:Stein}, we show that
a map $f\colon (Y,\cO_Y)\to (X,\cO_X)$ of Stein spaces is an open embedding if and only if
the respective homomorphism $f^\bullet\colon\cO(X)\to\cO(Y)$ is a $1$-pseudoflat
epimorphism. Some other equivalent homological conditions on $f^\bullet$ are also given.
This is a partial generalization of the main result of \cite{BBBK}.
However, in contrast to \cite{BBBK}, we work only over $\CC$, and we deal with
topological (rather than bornological) algebras. In Section~\ref{sect:C-inf}, we show that a similar
result holds for the algebras of $C^\infty$-functions on smooth real manifolds.
Section~\ref{sect:rem} contains some remarks and open questions related to function
algebras on Stein spaces and on $C^\infty$-differentiable spaces.

\section{Preliminaries}
\label{sect:prelim}

Throughout, all vector spaces and algebras are assumed to be over the field $\CC$
of complex numbers. All algebras are assumed to be associative and
unital.
By a {\em Fr\'echet algebra} we mean an algebra $A$ equipped with
a complete, metrizable locally convex topology
(i.e., $A$ is an algebra and a Fr\'echet space simultaneously)
such that the multiplication $A\times A\to A$ is continuous.
A {\em left Fr\'echet $A$-module} is a left $A$-module $M$ equipped with
a complete, metrizable locally convex topology in such a way that
the action $A\times M\to M$ is continuous.
We always assume that $1_A\cdot x=x$ for all $x\in M$, where $1_A$ is the identity of $A$.
Left Fr\'echet $A$-modules and their continuous morphisms form a category
denoted by $A\lmod$.
The categories $\rmod A$ and $A\bimod A$ of right
Fr\'echet $A$-modules and of Fr\'echet $A$-bimodules are defined similarly.
Note that $A\bimod A\cong A^e\lmod\cong\rmod A^e$, where $A^e=A\Ptens A^{\mathrm{op}}$,
and where $A^{\mathrm{op}}$ stands for the algebra opposite to $A$.
The space of morphisms from $M$ to $N$ in $A\lmod$
(respectively, in $\rmod A$, in $A\bimod A$) will be denoted by
${_A}\h(M,N)$ (respectively, $\h_A(M,N)$, ${_A}\h_A(M,N)$).
Given Fr\'echet algebras $A$ and $B$, we denote by $\Hom(A,B)$ the set of all
continuous algebra homomorphisms from $A$ to $B$.

If $M$ is a right Fr\'echet $A$-module and $N$
is a left Fr\'echet $A$-module, then their {\em $A$-module tensor product}
$M\ptens{A}N$ is defined to be
the quotient $(M\Ptens N)/L$, where $L\subset M\Ptens N$
is the closed linear span of all elements of the form
$x\cdot a\otimes y-x\otimes a\cdot y$
($x\in M$, $y\in N$, $a\in A$).\footnote[1]{Some authors (see, e.g.,
\cite{KV,T1,Ast,EschmPut}) define $M\ptens{A}N$ in a different way.
Actually, their $M\ptens{A}N$ is our $\Tor_0^A(M,N)$ (see below).
We adopt the definition given by M.~A.~Rieffel \cite{Rieffel} (see also \cite{X1,X2,CLM,Dales,Runde}).}
As in pure algebra, the $A$-module tensor product can be characterized
by the universal property that, for each Fr\'echet space $E$,
there is a natural bijection between the set of all
continuous $A$-balanced bilinear maps from $M\times N$ to $E$
and the set of all continuous linear maps from
$M\ptens{A}N$ to $E$.

A chain complex $C_\bullet=(C_n,d_n)_{n\in\Z}$ of Fr\'echet $A$-modules is {\em admissible} if
it splits in the category of topological vector spaces, i.e., if it has
a contracting homotopy consisting of continuous linear maps. Geometrically,
this means that $C_\bullet$ is exact, and $\Ker d_n$ is a complemented subspace
of $C_n$ for each $n$.
A left Fr\'echet $A$-module $P$ is {\em projective}
if the functor ${_A}\h(P,-)\colon A\lmod\to\Vect$
(where $\Vect$ is the category of vector spaces and linear maps)
is exact is the sense that it takes admissible sequences
of Fr\'echet $A$-modules to exact sequences of vector spaces.
Similarly, a left Fr\'echet $A$-module $F$ is {\em flat}
if the tensor product functor
$(-)\ptens{A} F\colon \rmod A\to\Vect$
is exact in the same sense as above.
It is known that every projective Fr\'echet module is flat.

A {\em projective resolution} of $M\in A\lmod$ is a pair $(P_\bullet,\eps)$
consisting of a nonnegative chain complex
$P_\bullet=(P_n,d_n)_{n\ge 0}$ in $A\lmod$ and a morphism
$\eps\colon P_0\to M$ such that the sequence
$0\lar M \xla{\eps} P_\bullet$ is an admissible complex
and such that all the modules $P_n$ ($n\ge 0$) are projective.
It is a standard fact that $A\lmod$ has {\em enough projectives},
i.e., each left Fr\'echet $A$-module has a projective resolution.
The same is true of $\rmod A$ and $A\bimod A$.
In particular, the (unnormalized)
{\em bimodule bar resolution} of $A$ \cite[Section III.2.3]{X1} looks as follows:
\begin{equation}
\label{barA}
0 \lar A \xla{\mu_A} A\Ptens A \xla{d} A\Ptens A\Ptens A \lar \cdots
\lar A^{\Ptens n} \lar \cdots
\end{equation}
Here $\mu_A$ is the multiplication map, and $d\colon A^{\Ptens 3}\to A^{\Ptens 2}$
is given by
\begin{equation}
\label{bar_d}
d(a\otimes b\otimes c)=ab\otimes c-a\otimes bc
\quad (a,b,c\in A).
\end{equation}
The explicit formula for the higher differentials $A^{\Ptens (n+1)}\to A^{\Ptens n}$
is similar [loc. cit.]; we do not need it here. The augmented complex \eqref{barA}
is a projective resolution of $A$ in $A\bimod A$.

If $M\in\rmod A$ and $N\in A\lmod$, then
the space $\Tor_n^A(M,N)$ is defined to be the $n$th
homology of the complex $P_\bullet\ptens{A} N$, where $P_\bullet$ is a projective
resolution of $M$. Equivalently, $\Tor_n^A(M,N)$ is the $n$th
homology of the complex $M\ptens{A} Q_\bullet$, where $Q_\bullet$ is a projective
resolution of $N$. The spaces $\Tor_n^A(M,N)$
do not depend on the particular choices of $P_\bullet$ and $Q_\bullet$
and have the usual functorial properties (see \cite[Section III.4.4]{X1} for details).
Note that $\Tor_n^A(M,N)$ is not necessarily Hausdorff, but the associated
Hausdorff space (i.e., the quotient of $\Tor_n^A(M,N)$ modulo the closure of zero)
is a Fr\'echet space.
If $M\in A\bimod A$, then the {\em $n$th
Hochschild homology} of $A$ with coefficients in $M$
is defined by $\cH_n(A,M)=\Tor_n^{A^e}(M,A)$.

In contrast to the purely algebraic case, $\Tor_0^A(M,N)$ is not the same as $M\ptens{A}N$.
Nevertheless, there is a natural continuous open linear surjection
\begin{equation}
\label{alpha_can}
\alpha_{M,N}\colon\Tor_0^A(M,N)\to M\ptens{A} N,
\end{equation}
whose kernel is the closure of zero in $\Tor_0^A(M,N)$
\cite[III.4.27]{X1}.
In other words, $M\ptens{A} N$ is isomorphic to the Hausdorff space associated
to $\Tor_0^A(M,N)$. Hence the following equivalences hold:
\begin{equation}
\label{Tor0-ptens}
\begin{split}
\Tor_0^A(M,N)\;\text{is Hausdorff}
&\iff\alpha_{M,N}\;\text{is injective}
\iff \alpha_{M,N}\;\text{is bijective}\\
&\iff \alpha_{M,N}\;\text{is a topological isomorphism}.
\end{split}
\end{equation}

Under some nuclearity assumptions, the derived functor $\Tor$
can be calculated with the help of exact (not necessarily admissible)
sequences of projective modules.
The following result is an easy modification
of \cite[Corollary 3.1.13]{EschmPut}
(which, in turn, goes back to \cite[Proposition 4.5]{T1}).

\begin{prop}
\label{prop:nucl_Tor}
Let $A$ be a Fr\'echet algebra, $M\in\rmod A$, and $N\in A\lmod$. Suppose that
\[
0 \lar M \lar P_0 \lar P_1 \lar \cdots \lar P_n\lar P_{n+1}
\]
is an exact sequence in $\rmod A$ such that $P_0,\ldots ,P_n$ are projective.
Assume that one of the following conditions holds:
\begin{mycompactenum}
\item
$P_0,\ldots ,P_{n+1}$ are nuclear;
\item
$A$ and $N$ are nuclear.
\end{mycompactenum}
Then for each $m=0,\ldots ,n$ the space $\Tor_m^A(M,N)$ is topologically
isomorphic to the $m$th homology of the complex
$P_\bullet\ptens{A}N$.
In particular, if either $M$ or $N$ is flat, then
the tensored sequence
\[
0 \lar M\ptens{A} N \lar P_0\ptens{A} N\lar  \cdots \lar P_{n+1}\ptens{A} N
\]
is exact.
\end{prop}

Given a Fr\'echet algebra $A$ and a Fr\'echet $A$-bimodule $M$,
we let $\Der(A,M)$ denote the space of all continuous derivations of
$A$ with values in $M$. The bimodule of {\em noncommutative differential
$1$-forms over $A$} is a Fr\'echet $A$-bimodule $\Omega^1 A$ together with a
derivation $d_A\colon A\to\Omega^1 A$ such that for each Fr\'echet $A$-bimodule $M$
and each derivation $D\colon A\to M$ there exists a unique $A$-bimodule
morphism $\Omega^1 A\to M$ making the following diagram commute:
\[
\xymatrix{
\Omega^1 A \ar[r] & M\\
A \ar[u]^{d_A} \ar[ur]_D
}
\]
In other words, we have a natural isomorphism
\[
{_A}\h_A(\Omega^1 A,M)\cong\Der(A,M) \qquad (M\in A\bimod A).
\]
It is a standard fact (see, e.g., \cite{CQ,Pir_qfree}) that $\Omega^1 A$ exists
and is isomorphic to the kernel of the multiplication map
$\mu_A\colon A\Ptens A\to A$. Under the above identification, the universal
derivation $d_A\colon A\to \Omega^1 A$ acts by the rule
$d_A(a)=1\otimes a-a\otimes 1$ ($a\in A$). Thus we have an exact sequence
\begin{equation}
\label{Omega}
0 \to \Omega^1 A \xra{j_A} A\Ptens A \xra{\mu_A} A \to 0
\end{equation}
in $A\bimod A$, where $j_A$ is uniquely determined by $j_A(d_A(a))=1\otimes a-a\otimes 1$
($a\in A$). Note that \eqref{Omega} splits in $A\lmod$ and in $\rmod A$
(\cite{Pir_qfree}, cf. also \cite{CQ}). In particular, \eqref{Omega} is admissible.

\section{Pseudoflat epimorphisms}
\label{sect:psd}

We begin this section with the following ``truncated'' version of the transversality relation
$\perp_A$ introduced in \cite{KV} (see also \cite{Ast,EschmPut,Demailly}).

\begin{prop}
\label{pstranschar}
Let $A$ be a Fr\'{e}chet algebra, $M\in \rmod A$, $N\in
A\lmod$, and $n\in \Z_+$. Then the following conditions are equivalent:
\begin{mycompactenum}
\item
$\Tor_m^A(M,N) = 0$ for $1\le m\le n$, and $\Tor_0^A(M,N)$ is
Hausdorff;
\item
for some (or, equivalently, for each)  projective resolution
$0 \lar M \lar P_\bullet$ in $\rmod A$ the sequence
\begin{equation}
\label{PN}
 0 \lar M\ptens{A} N \lar P_0\ptens{A}  N \lar \cdots \lar  P_{n+1}\ptens{A} N
\end{equation}
is exact;
\item
for some (or, equivalently, for each)  projective resolution
$0 \lar N \lar Q_\bullet$ in $A\lmod$ the sequence
\begin{equation*}
0 \lar M\ptens{A} N \lar M\ptens{A}  Q_0  \lar \cdots \lar  M\ptens{A}
Q_{n+1}
\end{equation*}
is exact;
\item
for some (or, equivalently, for each)  projective resolution
$0 \lar A \lar L_\bullet$ in $A\bimod A$ the sequence
\begin{equation}
\label{MLN}
0 \lar M\ptens{A} N \lar M\ptens{A} L_0\ptens{A} N   \lar \cdots
 \lar M\ptens{A} L_{n+1}\ptens{A} N
\end{equation}
is exact.
\end{mycompactenum}
\end{prop}
\begin{proof}
The equivalences between ``for some'' and ``for each'' in (ii)--(iv) are immediate from
the fact that all projective resolutions of a module are homotopy equivalent.

$\mathrm{(i)}\Longleftrightarrow\mathrm{(ii)}$. Since $\ptens{A}$ preserves surjections
\cite[II.4.12]{X1}, \eqref{PN} is always exact at $M\ptens{A} N$.
If $1\le m\le n$, then \eqref{PN} is exact at $P_m\ptens{A}N$ if and only if
$\Tor_m^A(M,N)=0$. On the other hand, \eqref{PN} is exact at $P_0\ptens{A}N$ if and only if
\[
\Ker(P_0\ptens{A} N\to\Tor_0^A(M,N))=\Ker(P_0\ptens{A} N\to M\ptens{A} N),
\]
i.e., if and only if the canonical map $\Tor_0^A(M,N)\to M\ptens{A} N$ is injective.
By \eqref{Tor0-ptens}, the latter condition holds if and only if $\Tor_0^A(M,N)$ is Hausdorff.

$\mathrm{(i)}\Longleftrightarrow\mathrm{(iii)}$. This is similar to $\mathrm{(i)}\iff\mathrm{(ii)}$.

$\mathrm{(iii)}\Longleftrightarrow\mathrm{(iv)}$.
If $0\lar A\lar L_\bullet$ is a projective resolution of $A$ in $A\bimod A$,
then $0\lar N \lar L_\bullet\ptens{A} N$ is a projective resolution of $N$
in $A\lmod$. The rest is clear.
\end{proof}

\begin{definition}
\label{ntrans}
Let $A$ be a Fr\'{e}chet algebra, $M\in \rmod A$,
$N\in A\lmod$, and $n\in \Z_+$. We say that $M$ and $N$ are
{\em $n$-transversal} over $A$ (and write $M\perp^n_A N$) if the (equivalent)
conditions of Proposition~\ref{pstranschar} are satisfied.
If $M\perp^n_A N$ for all $n\in\Z_+$, then $M$ and $N$
are said to be {\em transversal} \cite{KV}
(see also \cite{Ast,EschmPut,Demailly}). In this case, we write $M\perp_A N$.
\end{definition}

\begin{corollary}
Let $A$ be a Fr\'{e}chet algebra, $M\in \rmod A$, $N\in
A\lmod$, and $n\in \Z_+$. Then the following conditions are equivalent:
\begin{mycompactenum}
\item
$M\perp^n_A N$;
\item
$(N\Ptens M)\perp^n_{A^e} A$;
\item
$\cH_m(A,N\Ptens M) = 0$ for $1\le m\le n$, and $\cH_0(A,N\Ptens M)$ is
Hausdorff.
\end{mycompactenum}
\end{corollary}
\begin{proof}
This is immediate from Proposition \ref{pstranschar} and from the
isomorphisms $\Tor_m^A(M,N)\cong\cH_m(A,N\Ptens M)$ \cite[III.4.25]{X1}.
\end{proof}

Here is our main definition.

\begin{definition}
\label{def:pseudoflat}
Let $\varphi\colon A\to B$ be a Fr\'echet algebra homomorphism, and let $n\in\Z_+$.
We say that $\varphi$ is {\em $n$-pseudoflat} if $B\perp^n_A B$.
\end{definition}

We are mostly interested in those pseudoflat homomorphisms which are epimorphisms
(in the category-theoretic sense).
Recall that a morphism $\varphi\colon A\to B$ in a category $\cC$ is an
{\em epimorphism} if for each pair $\psi,\psi'\colon B\to C$ of morphisms in $\cC$
the equality $\psi\varphi=\psi'\varphi$ implies that $\psi=\psi'$.
Equivalently, this means that, for each object $C$ of $\cC$, the map
$\Hom_{\cC}(B,C)\to\Hom_{\cC}(A,C)$ induced by $\varphi$ is injective.

For the reader's convenience, let us recall the following
result (see, e.g., \cite[Prop.~XI.1.2]{St}, \cite[Prop. 6.1]{Pir_qfree}).

\begin{prop}\label{charepi}
Let $\varphi\colon A\to B$ be a homomorphism of Fr\'echet algebras. Then the
following conditions are equivalent:
\begin{mycompactenum}
\item $\varphi$ is an epimorphism in the category of Fr\'echet algebras;
\item  the multiplication map $\mu_{B,A}\colon B\ptens{A}B \to B$ is a topological isomorphism;
\item for each $M\in \rmod B$ and each $N\in B\lmod$,
the canonical map $M \ptens{A} N \to M \ptens{B} N$ is a topological
isomorphism.
\end{mycompactenum}
\end{prop}

We also need the following well-known fact (see, e.g., \cite[Chap. 0, Corollary 4.2]{X1}).

\begin{prop}
\label{prop:epi_mod}
Let $A$ be a Fr\'echet algebra, and let $\varphi\colon M\to N$ be a morphism
of left Fr\'echet $A$-modules. Then $\varphi$ is an epimorphism in $A\lmod$
if and only if $\varphi(M)$ is dense in $N$. The same result holds
for the categories $\rmod A$ and $A\bimod A$.
\end{prop}

Given a Fr\'echet algebra homomorphism $\varphi\colon A\to B$, we
define $\bar\mu_{B,A}\colon \Tor_0^A(B,B)\to B$ to be the composition of the
canonical map $\alpha=\alpha_{B,B}\colon\Tor_0^A(B,B)\to B\ptens{A} B$ (see \eqref{alpha_can})
and the multiplication map $\mu_{B,A}\colon B\ptens{A} B\to B$.

\begin{lemma}
\label{lemma:L1}
The following conditions are equivalent:
\begin{mycompactenum}
\item
$\varphi$ is a $0$-pseudoflat epimorphism;
\item
$\bar\mu_{B,A}\colon \Tor_0^A(B,B)\to B$ is bijective;
\item
$\bar\mu_{B,A}\colon \Tor_0^A(B,B)\to B$ is a topological isomorphism.
\end{mycompactenum}
\end{lemma}
\begin{proof}
$\mathrm{(i)}\Longrightarrow\mathrm{(iii)}$.
The fact that $\varphi$ is $0$-pseudoflat means precisely that $\Tor_0^A(B,B)$ is
Hausdorff, which happens if and only if $\alpha\colon\Tor_0^A(B,B)\to B\ptens{A} B$
is a topological isomorphism (see \eqref{Tor0-ptens}).
On the other hand, the fact that $\varphi$ is an epimorphism
means precisely that $\mu_{B,A}\colon B\ptens{A} B\to B$ is a topological isomorphism
(see Proposition~\ref{charepi}). Hence so is $\bar\mu_{B,A}=\mu_{B,A}\circ\alpha$.

$\mathrm{(iii)}\Longrightarrow\mathrm{(ii)}$. This is clear.

$\mathrm{(ii)}\Longrightarrow\mathrm{(i)}$.
Since $\bar\mu_{B,A}$ is continuous and bijective, we conclude that $\Tor_0^A(B,B)$
is Hausdorff (i.e., $\varphi$ is $0$-pseudoflat).
Hence $\alpha$ is bijective by \eqref{Tor0-ptens}, and so $\mu_{B,A}$ is a topological
isomorphism by the Open Mapping Theorem.
Applying Proposition~\ref{charepi}, we see that $\varphi$ is an epimorphism.
\end{proof}

Let $\varphi\colon A\to B$ be a Fr\'echet algebra homomorphism,
and let $0 \lar A \lar L_\bullet$ be a projective
resolution in $A\bimod A$.
Applying $B\ptens{A} (-)\ptens{A} B$ to
$L_0\to A$ and composing with the multiplication $\mu_{B,A}\colon B\ptens{A} B\to B$,
we get a $B$-bimodule morphism $\eps_L\colon B\ptens{A} L_0\ptens{A} B\to B$.
Similarly, if $0\lar B \lar P_\bullet$ and $0\lar B \lar Q_\bullet$ are projective resolutions
of $B$ in $\rmod A$ and in $A\lmod$, respectively, then we have morphisms
$\eps_P\colon P_0\ptens{A} B\to B$ in $\rmod B$ and $\eps_Q\colon B\ptens{A} Q_0\to B$ in $B\lmod$.

\begin{prop}
\label{psflepi}
Let $\varphi\colon A\to B$ be a Fr\'echet algebra homomorphism.
The following conditions are equivalent:
\begin{mycompactenum}
\item
$\varphi$ is an $n$-pseudoflat epimorphism;
\item
for some (or, equivalently, for each)  projective resolution
$0 \lar B \lar P_\bullet$ in $\rmod A$ the sequence
\begin{equation}
\label{PB}
 0 \lar B \xla{\eps_P} P_0\ptens{A}  B \lar \cdots \lar  P_{n+1}\ptens{A} B
\end{equation}
is exact;
\item
for some (or, equivalently, for each)  projective resolution
$0 \lar B \lar Q_\bullet$ in $A\lmod$ the sequence
\begin{equation*}
0 \lar B \xla{\eps_Q} B\ptens{A}  Q_0  \lar \cdots \lar  B\ptens{A} Q_{n+1}
\end{equation*}
is exact;
\item
for some (or, equivalently, for each)  projective resolution
$0 \lar A \lar L_\bullet$ in $A\bimod A$ the sequence
\begin{equation}
\label{qflres}
0 \lar B \xla{\eps_L} B\ptens{A} L_0\ptens{A} B   \lar \cdots
 \lar B\ptens{A} L_{n+1}\ptens{A} B
\end{equation}
is exact.
\end{mycompactenum}
\end{prop}
\begin{proof}
$\mathrm{(i)}\Longleftrightarrow\mathrm{(ii)}$.
Clearly, if $m\ge 1$, then \eqref{PB} is exact at $P_m\ptens{A} B$ if and only if
$\Tor_m^A(B,B)=0$. Since the $0$th homology of $P_\bullet\ptens{A} B$ is precisely
$\Tor_0^A(B,B)$, we see that \eqref{PB} is exact at $P_0\ptens{A} B$ if and only if
$\bar\mu_{B,A}\colon\Tor_0^A(B,B)\to B$ is bijective.
Now the result follows from Lemma~\ref{lemma:L1}.

Equivalences $\mathrm{(i)}\Longleftrightarrow\mathrm{(iii)}$ and
$\mathrm{(i)}\Longleftrightarrow\mathrm{(iv)}$ are proved similarly.
\end{proof}

\begin{corollary}
\label{cor:0psd}
Let $\varphi\colon A\to B$ be a Fr\'echet algebra homomorphism.
Define
\[
d_\varphi\colon B\Ptens A\Ptens B\to B\Ptens B, \quad
b\otimes a\otimes c\mapsto b\varphi(a)\otimes c-b\otimes\varphi(a)c
\quad (b,c\in B,\; a\in A).
\]
Then $\varphi$ is a $0$-pseudoflat
epimorphism if and only if the sequence
\begin{equation}
\label{0psd}
0 \lar B \xla{\mu_B} B\Ptens B \xla{d_\varphi} B\Ptens A\Ptens B
\end{equation}
is exact.
\end{corollary}
\begin{proof}
If $0\lar A\lar L_\bullet$ is the bimodule bar resolution of $A$ (see \eqref{barA}),
then \eqref{qflres} for $n=0$ is precisely \eqref{0psd}.
\end{proof}

\begin{corollary}
\label{cor:surj_0psd}
A surjective Fr\'echet algebra homomorphism is a $0$-pseudoflat epimorphism.
\end{corollary}
\begin{proof}
If we replace $A$ by $B$ in \eqref{0psd}, then we get an exact sequence
(in fact, this is the low-dimensional segment of the bimodule bar resolution for $B$).
Since $\varphi\colon A\to B$ is onto, we conclude that \eqref{0psd}
is exact as well.
\end{proof}

\begin{remark}
As we shall see below (Example \ref{example:non_0psd}), a Fr\'echet algebra
homomorphism with dense image, while being an epimorphism for an obvious reason,
 is not necessarily $0$-pseudoflat.
\end{remark}

The next proposition (which is a Fr\'echet algebra version of \cite[(87)]{BerDic})
emphasizes the difference between $0$-pseudoflat
and $1$-pseudoflat epimorphisms.

\begin{prop}
\label{prop:quot_1psd}
Let $I$ be a closed two-sided ideal in a nuclear Fr\'{e}chet algebra $A$.
Then the quotient map $\pi\colon A\to A/I$ is a $1$-pseudoflat epimorphism if
and only if the multiplication map
$\mu_I\colon I\Ptens I\to I$, $a\otimes b\mapsto ab$, is surjective.
\end{prop}
\begin{proof}
By Corollary \ref{cor:surj_0psd}, $\pi$ is a $0$-pseudoflat epimorphism.
Thus $\pi$ is a $1$-pseudoflat epimorphism if and only if $\Tor_1^A(A/I,A/I)=0$.
Since $A$ is nuclear, the exact sequence
\begin{equation}
\label{AI-short}
0 \lar A/I \lar A \lar I \lar 0
\end{equation}
induces a long exact sequence for $\Tor_i^A(A/I,-)$
(see \cite[Theorem 3.1.12]{EschmPut}), whose low-dimensional
segment looks as follows:
\begin{equation}
\label{Tor_AI_AI}
0 \lar \Tor_0^A(A/I,A/I) \xla{q} \Tor_0^A(A/I,A) \lar \Tor_0^A(A/I,I) \lar \Tor_1^A(A/I,A/I) \lar 0.
\end{equation}
Clearly, $\Tor_0^A(A/I,A)\cong (A/I)\ptens{A} A\cong A/I$. Applying
Corollary \ref{cor:surj_0psd} and Lemma~\ref{lemma:L1}, we see that
$\Tor_0^A(A/I,A/I)\cong A/I$. Under the above identifications,
the map $q$ in \eqref{Tor_AI_AI} becomes the identity map of $A/I$.
Hence $\Tor_1^A(A/I,A/I)$ is isomorphic to $\Tor_0^A(A/I,I)$.
Applying Proposition~\ref{prop:nucl_Tor} to \eqref{AI-short}, we see that
$\Tor_0^A(A/I,I)$ is the cokernel of $I\ptens{A} I\to A\ptens{A} I\cong I$,
which is isomorphic to the cokernel of $\mu_I$. The rest is clear.
\end{proof}

\begin{example}
Let $A=\cO(\CC)$ be the algebra of holomorphic functions on $\CC$,
and let $I=\{ f\in A : f(0)=0\}$. The multiplication map
$\mu_I\colon I\Ptens I\to I$ is not surjective, because the image of $\mu_I$
is contained in the ideal
$J=\{ f\in A : f(0)=f'(0)=0\}$, which is strictly smaller than $I$.
Hence the quotient map $A\to A/I$ is not $1$-pseudoflat by Proposition~\ref{prop:quot_1psd},
although it is a $0$-pseudoflat epimorphism by Corollary~\ref{cor:surj_0psd}.
\end{example}

\begin{definition}
\label{def:homepi}
Let $\varphi\colon A\to B$ be a Fr\'echet algebra homomorphism.
We say that $\varphi$ is a {\em weak homological epimorphism} if
$\varphi$ is an $n$-pseudoflat epimorphism for all $n\in\Z_+$.
\end{definition}

Thus $\varphi\colon A\to B$ is a weak homological epimorphism if and only if
any (hence all) of the infinite sequences
\begin{equation}
\label{PQL}
0 \lar B \lar P_\bullet\ptens{A} B, \quad
0 \lar B \lar B\ptens{A} Q_\bullet, \quad
0 \lar B \lar B\ptens{A} L_\bullet\ptens{A} B
\end{equation}
(where $P_\bullet$, $Q_\bullet$, $L_\bullet$ are as in Proposition~\ref{psflepi}) are exact.

\begin{definition}
We say that $\varphi$ is a {\em strong homological epimorphism} if
any (hence all) of the infinite sequences \eqref{PQL} are admissible.
\end{definition}

The fact that the admissibility of any of the sequences \eqref{PQL} implies the
admissibility of the other two follows from \cite[Prop. 3.2]{Pir_stbflat}.

\begin{remark}
\label{rem:hist}
The notion of a homological epimorphism has a remarkable history.
Strong homological epimorphisms were introduced by J.~L.~Taylor \cite{T2} under
the name of ``absolute localizations''. For nuclear Fr\'echet algebras, our notion
of a weak homological epimorphism is equivalent to Taylor's notion of a ``localization''
[loc. cit.]; see Section~\ref{sect:intro}.
In the purely algebraic setting, homological epimorphisms were rediscovered
by W.~Dicks \cite{Dicks} under the name of ``liftings'',
by W.~Geigle and H.~Lenzing \cite{GL} (where the current terminology was introduced),
by A.~Neeman and A.~Ranicki \cite{NR} under the name of ``stably flat homomorphisms''.
In \cite{Meyer}, R.~Meyer introduced strong homological epimorphisms in the setting
of nonunital bornological algebras under the name of ``isocohomological morphisms''.
Finally, O.~Ben-Bassat and K.~Kremnizer \cite{BBK} introduced weak homological epimorphisms
(under the name of ``homotopy epimorphisms'') in the abstract setting of commutative
algebras in symmetric monoidal quasi-abelian categories (cf. also \cite{HAGII}).
Amazingly, each of the above-mentioned authors seems to have introduced
essentially the same class of morphisms independently of the earlier literature.
\end{remark}

The following proposition is an analog of \cite[Prop.~5.1]{BerDic}.

\begin{prop}\label{psflnuc}
Let  $\varphi\colon A\to B$ be a Fr\'echet algebra epimorphism, and let $n\in\Z_+$.
Suppose that $A$ and $B$ are nuclear. Then the following conditions
are equivalent:
\begin{mycompactenum}
\item
$\varphi$ is $n$-pseudoflat;
\item
$M\perp^n_A B$ for each right Fr\'{e}chet $B$-module $M$;
\item
$B\perp^n_A N$ for each  left Fr\'{e}chet $B$-module $N$.
\end{mycompactenum}
\end{prop}
\begin{proof}
$\mathrm{(i)}\Longrightarrow\mathrm{(iii)}$.
Let $0 \lar B \lar P_\bullet$ be a projective resolution in $\rmod A$ such that
all the modules $P_i$ are nuclear. By Proposition~\ref{psflepi},
\eqref{PB} is an exact sequence. Observe that all the modules in~\eqref{PB}
(including $B$) are nuclear and projective in $\rmod B$.
Therefore, applying $(-)\ptens{B}N$ and using Proposition~\ref{prop:nucl_Tor},
we get an exact sequence
\begin{equation}
\label{PN2}
0 \lar N \lar P_0\ptens{A} N \lar \cdots \lar P_{n+1}\ptens{A} N.
\end{equation}
Since $\varphi$ is an epimorphism, we see that $B\ptens{A} N\cong N$ canonically.
Hence \eqref{PN2} is isomorphic to \eqref{PN} with $M=B$.
Thus $B\perp^n_A N$.

The implication $\mathrm{(i)}\Longrightarrow\mathrm{(ii)}$ is proved
similarly;  $\mathrm{(ii)}\Longrightarrow\mathrm{(i)}$ and
$\mathrm{(iii)}\Longrightarrow\mathrm{(i)}$ are clear from
Definition~\ref{def:pseudoflat}.
\end{proof}

\begin{corollary}
Let  $\varphi\colon A\to B$ be a Fr\'echet algebra epimorphism.
Suppose that $A$ and $B$ are nuclear. Then the following conditions
are equivalent:
\begin{mycompactenum}
\item
$\varphi$ is a weak homological epimorphism;
\item
$M\perp_A B$ for each right Fr\'{e}chet $B$-module $M$;
\item
$B\perp_A N$ for each  left Fr\'{e}chet $B$-module $N$.
\end{mycompactenum}
\end{corollary}

\begin{remark}
A similar result
\cite[Prop.~3.2]{Pir_stbflat} on strong homological epimorphisms does not
involve nuclearity assumptions.
\end{remark}

Our next goal is to characterize epimorphisms, $0$-pseudoflat epimorphisms,
and $1$-pseudoflat epimorphisms in terms of noncommutative differential forms.
Towards this goal, let us introduce some notation.
Let $\varphi\colon A\to B$ be a Fr\'echet algebra homomorphism.
Applying $B\ptens{A}(-)\ptens{A} B$ to the canonical sequence \eqref{Omega}
and composing with $B\ptens{A} B\to B$, we get
\[
0 \to B\ptens{A}\Omega^1 A\ptens{A} B \to B\Ptens B \xra{\mu_B} B \to 0.
\]
Identifying $\Omega^1 B$ with $\Ker\mu_B$, we see that there exists
a unique $B$-bimodule
morphism $\check\varphi\colon B\ptens{A}\Omega^1 A\ptens{A} B \to\Omega^1 B$
making the diagram
\begin{equation}
\label{check_varphi}
\xymatrix{
& B\ptens{A}\Omega^1 A\ptens{A} B \ar[r]  \ar[d]_{\check\varphi}
& B\Ptens B \ar@{=}[d] \ar[r]^(.6){\mu_B} & B \ar@{=}[d] \ar[r] & 0\\
0 \ar[r] & \Omega^1 B \ar[r]^{j_B} & B\Ptens B \ar[r]^(.6){\mu_B} & B \ar[r] & 0
}
\end{equation}
commute.

For each Fr\'echet $B$-bimodule $X$ we have a linear map
$$
\widetilde\varphi_X\colon\Der (B, X)\to \Der (A, X),\quad D\mapsto
D\varphi.
$$

\begin{theorem}\label{epicharde}
For a Fr\'echet algebra homomorphism $\varphi\colon A\to B$ the following
conditions are equivalent:
\begin{mycompactenum}
\item
$\varphi$ is an epimorphism;
\item
$\widetilde\varphi_X\colon\Der (B, X)\to \Der (A, X)$ is injective for each
$X\in B\bimod B$;
\item
the image of $\check\varphi\colon B\ptens{A} \Omega^1 A  \ptens{A} B \to \Omega^1 B$ is
dense in $\Omega^1 B$.
\end{mycompactenum}
\end{theorem}
\begin{proof}
$\mathrm{(i)}\Longrightarrow\mathrm{(ii)}$.
Given $X\in B\bimod B$, we make the Fr\'echet space
$B\oplus X$ into a Fr\'echet algebra by letting
\[
(b,x)(c,y)=(bc,by+xc)\qquad (b,c\in B,\; x,y\in X).
\]
Suppose that $\varphi$ is an
epimorphism. For every $D\in\Der (B, X)$ we have a Fr\'echet algebra homomorphism
\[
\psi\colon B\to B\oplus X,\quad \psi(b)=(b,D(b))\quad (b\in B).
\]
If $D\varphi=0$, then $\psi\varphi=\psi'\varphi$, where
\[
\psi'\colon B\to B\oplus X,\quad \psi'(b)=(b,0)\quad (b\in B).
\]
Since $\varphi$  is an epimorphism, we have $\psi=\psi'$, i.e., $D=0$.

$\mathrm{(ii)}\Longrightarrow\mathrm{(i)}$. Suppose that $\widetilde\varphi_X$ is
injective for each $X\in B\bimod B$. Consider the derivation
$$
D\colon B\to B\ptens{A} B,\quad b\mapsto b\otimes_A 1-1\otimes_A b.
$$
Since $D\varphi=0$, we have $D=0$, i.e., $b\otimes_A 1=1\otimes_A b$ for each $b\in B$. Then
the continuous linear map $B\to B\ptens{A} B$, $b\mapsto b\otimes_A 1$, is the
inverse of the multiplication $\mu_{B,A}\colon B\ptens{A} B\to B$. Thus $\mu_{B,A}$ is
a topological isomorphism, i.e., $\varphi$ is an epimorphism (see Proposition~\ref{charepi}).

$\mathrm{(ii)}\Longleftrightarrow\mathrm{(iii)}$. By the universal
properties of $\Omega^1$ and $\ptens{A}$, for each Fr\'echet
$B$-bimodule $X$  there exists a commutative diagram
\begin{equation}\label{derprdi}
\xymatrix@C=20pt{
 \Der(B,X) \ar[rr]^{\wt{\varphi}_X}\ar@{=}[d] && \Der(A,X) \ar@{=}[d]  \\
 {_B}\h{_B}( \Omega^1 B,X)\ar[rr]^(.4){\check\varphi_X}
 &&{_B}\h{_B}(B\ptens{A} \Omega^1 A  \ptens{A}
B,X)
 }
\end{equation}
where $\check\varphi_X$ induced by $\check\varphi$.
Thus (ii) holds if and only if
$\check\varphi$ is an epimorphism in $B\bimod B$, which is equivalent to (iii)
by Proposition~\ref{prop:epi_mod}.
\end{proof}

\begin{theorem}
\label{derpr0pf}
A Fr\'echet algebra homomorphism $\varphi\colon A\to B$ is a $0$-pseudoflat epimorphism
if and only if $\check\varphi\colon B\ptens{A} \Omega^1 A  \ptens{A} B \to \Omega^1 B$ is onto.
\end{theorem}
\begin{proof}
Since $j_A$ is a kernel of $\mu_A$ in \eqref{Omega},
we see that the map $d\colon A^{\Ptens 3}\to A^{\Ptens 2}$
defined by \eqref{bar_d} factorizes
as follows:
\[
\xymatrix{
& A\Ptens A\Ptens A \ar[d]_{p_A} \ar[dr]^{d}\\
0 \ar[r] & \Omega^1 A \ar[r]^{j_A} & A\Ptens A \ar[r]^(.6){\mu_A} & A \ar[r] & 0.
}
\]
Applying $B\ptens{A}(-)\ptens{A} B$ and combining with \eqref{check_varphi},
we get the commutative diagram
\[
\xymatrix{
& B\Ptens A\Ptens B \ar[d]_{\tilde p_A} \ar[dr]^{d_\varphi}\\
& B\ptens{A}\Omega^1 A\ptens{A} B \ar[r] \ar[d]_{\check\varphi}
& B\Ptens B \ar[r]^(.6){\mu_B} \ar@{=}[d] & B \ar[r] \ar@{=}[d] & 0\\
0 \ar[r] & \Omega^1 B \ar[r]^{j_B} & B\Ptens B \ar[r]^(.6){\mu_B} & B \ar[r] & 0
}
\]
where $d_\varphi$ is defined in Corollary~\ref{cor:0psd}.
Since $j_B$ is a kernel of $\mu_B$, we see that \eqref{0psd} is exact if and only if
$\check\varphi\circ\tilde p_A$ is onto. Since $p_A$ is onto, and since the projective
tensor product preserves surjections of Fr\'echet modules, it follows that
$\tilde p_A$ is onto. Hence $\check\varphi\circ\tilde p_A$ is onto if and only if $\check\varphi$ is onto.
This completes the proof.
\end{proof}

To construct examples of epimorphisms that are not $0$-pseudoflat,
we need the following lemma.

\begin{lemma}
\label{lemma:Tor0}
Let $\varphi\colon A\to B$ be a $0$-pseudoflat Fr\'echet algebra epimorphism,
let $M\in\rmod B$, and let $N\in B\lmod$. Assume that either $M$ or $N$ is flat
as a Fr\'echet $B$-module.
Then $\Tor_0^A(M,N)$ is Hausdorff.
\end{lemma}
\begin{proof}
Since $\varphi$ is a $0$-pseudoflat epimorphism, we see that \eqref{0psd} is exact.
Since $j_B$ is a kernel of $\mu_B$, there exists a surjective morphism
$\bar\varphi\colon B \Ptens A\Ptens B\to \Omega^1 B$ such that the diagram
\begin{equation}
\label{bar-Omega1}
\xymatrix{
0 & B \ar[l] & B\Ptens B \ar[l]_{\mu_B} && B \Ptens A\Ptens B \ar[ll] \ar[dl]^{\bar\varphi}\\
&&& \Omega^1 B \ar[ul]^{j_B}
}
\end{equation}
commutes. (Note that $\bar\varphi=\check\varphi\circ\tilde p_A$, see the
proof of Proposition~\ref{derpr0pf}.)

Assume now that $N\in B\lmod$ is flat. Since the canonical sequence
\begin{equation}
\label{can_B}
0 \lar B \xla{\mu_B} B\Ptens B \xla{j_B} \Omega^1 B \lar 0
\end{equation}
splits in $B\lmod$, for each $M\in\rmod B$ the sequence
\[
0 \lar M \lar M\Ptens B \lar M\ptens{B}\Omega^1 B \lar 0
\]
obtained from \eqref{can_B} via $M\ptens{B}(-)$ is admissible. Applying $(-)\ptens{B} N$,
we get an exact sequence
\begin{equation}
\label{MN-Omega1}
0 \lar M\ptens{B} N \lar M\Ptens N \lar M\ptens{B}\Omega^1 B\ptens{B} N \lar 0.
\end{equation}
Let us now apply $M\ptens{B} (-) \ptens{B} N$ to \eqref{bar-Omega1}.
We obtain the following commutative diagram:
\begin{equation}
\label{bar-Omega1-MN}
\xymatrix@C-5pt{
0 & M\ptens{B} N \ar[l] & M\Ptens N \ar[l] && M \Ptens A\Ptens N \ar[ll] \ar[dl]^{\bar\varphi_{M,N}}\\
&&& M\ptens{B} \Omega^1 B \ptens{B} N \ar[ul]
}
\end{equation}
Since $\bar\varphi$ is onto, it follows that $\bar\varphi_{M,N}=\id_M\tens{B}\bar\varphi\tens{B}\id_N$
is onto. Together with the exactness of \eqref{MN-Omega1}, this implies that the
upper row of \eqref{bar-Omega1-MN} is exact. Identifying $M\ptens{B} N$ with $M\ptens{A} N$
(see Proposition~\ref{charepi}), we see that the sequence
\[
0 \lar M\ptens{A} N \lar M\Ptens N \lar M\Ptens A\Ptens N
\]
obtained from the low-dimensional segment of \eqref{barA} via $M\ptens{A} (-) \ptens{A} N$ is exact.
Equivalently, this means that the canonical map $\Tor_0^A(M,N)\to M\ptens{A} N$
is bijective, which happens if and only if $\Tor_0^A(M,N)$ is Hausdorff (see \eqref{Tor0-ptens}).

In the case where $N$ is arbitrary and $M$ is flat, the proof is similar.
\end{proof}

\begin{example}
\label{example:non_0psd}
Using Lemma~\ref{lemma:Tor0}, it is easy to construct Banach algebra
epimorphisms that are not $0$-pseudoflat. Consider, for example, the
nonunital Banach sequence algebras $\ell^1$ and $c_0$ (under pointwise
multiplication), let $A=\ell^1_+$ and $B=(c_0)_+$ denote their unitizations,
and let $\varphi\colon A\to B$ be the tautological embedding.
Since $\varphi(A)$ is dense in $B$, we see that $\varphi$ is an epimorphism.
Assume, towards a contradiction, that $\varphi$ is $0$-pseudoflat.
By \cite{X_flatmod}, the $1$-dimensional $B$-module $\CC=B/c_0$ is flat
(in fact, all Banach $B$-modules are flat \cite[VII.2.29]{X1}, because
$B$ is amenable \cite{Johnson}). Hence Lemma~\ref{lemma:Tor0} implies
that $\Tor_0^A(\CC,c_0)$ is Hausdorff. Consider now the admissible sequence
\[
0 \to \ell^1 \to A \to \CC \to 0
\]
of Banach $A$-modules (where $\ell^1\to A$ is the tautological embedding).
The low-dimensional segment of the respective long exact sequence
for $\Tor_i^A(-,c_0)$ looks as follows:
\begin{equation}
\label{Tor_c0}
0 \lar \Tor_0^A(\CC,c_0) \lar c_0 \xla{j} \Tor_0^A(\ell^1,c_0) \lar \Tor_1^A(\CC,c_0) \lar 0.
\end{equation}
By \cite[IV.5.9]{X1}, $\ell^1$ is a biprojective Banach algebra
(i.e., $\ell^1$ is a projective Banach $\ell^1$-bimodule), which implies, in particular,
that $\ell^1$ is projective in $\rmod A$ \cite[IV.1.3]{X1}.
Hence we may identify $\Tor_0^A(\ell^1,c_0)$ with
$\ell^1\ptens{A} c_0=\ell^1\ptens{\ell^1} c_0$, which is isomorphic to $\ell^1$
via the map $a\otimes x\mapsto (a_n x_n)$ (cf. \cite[II.3.9]{X1}
or \cite[Lemma 4.1]{Pir_bipr}). Under this identification, the map $j$ in \eqref{Tor_c0}
is nothing but the embedding of $\ell^1$ into $c_0$. This implies that
$\Tor_0^A(\CC,c_0)$ is topologically isomorphic to $c_0/\ell^1$ and is therefore
non-Hausdorff. The resulting contradiction shows that $\varphi$ is not $0$-pseudoflat.
\end{example}

Our next theorem is a functional analytic version of a result obtained by
Bergman and Dicks \cite[Remark 5.4]{BerDic} in the purely algebraic context.
The statement involves the following
condition\footnote[1]{This condition is missing in the first version of this preprint,
as well as in its journal version [J. Math. Anal. Appl. 485 (2020) 123817].
As a result, the proof of Theorem~\ref{kalch} given there contains a gap.
Specifically, the proof refers to Lemma~\ref{lemma:Tor0}, which actually does not apply
to the right $A$-module $B\ptens{A} \Omega^1 A$ because is has no
right $B$-module structure in general. This is the reason why we have to introduce
condition $(*)$ here. We do not know, however, whether this condition is really essential.

A corrigendum to [J. Math. Anal. Appl. 485 (2020) 123817] will presumably be published
in the same journal.}
on a Fr\'echet algebra homomorphism $A\to B$:
\begin{itemize}
\item[$(*)$]
$\Tor_0^A(B\ptens{A}\Omega^1 A,B)$ is Hausdorff.
\end{itemize}

\begin{theorem}
\label{kalch}
For a Fr\'echet algebra homomorphism $\varphi\colon A\to B$ the following
conditions are equivalent:
\begin{mycompactenum}
\item
$\varphi$ is a $1$-pseudoflat epimorphism;
\item
$\widetilde\varphi_X\colon\Der (B, X)\to \Der (A, X)$ is bijective for each
$X\in B\bimod B$, and $(*)$ holds;
\item
$\check\varphi\colon B\ptens{A} \Omega^1 A  \ptens{A} B \to \Omega^1 B$ is an
isomorphism in $B\bimod B$, and $(*)$ holds.
\end{mycompactenum}
\end{theorem}
\begin{proof}
Since the canonical sequence \eqref{Omega} splits in $A\lmod$, the sequence
\begin{equation}
\label{B_tens_Omega}
0 \lar B \lar B\Ptens A \lar B\ptens{A}\Omega^1 A \lar 0
\end{equation}
obtained from \eqref{Omega} via $B\ptens{A}(-)$ is admissible in $\rmod A$.
Since $B\Ptens A$ is projective in $\rmod A$, the low-dimensional segment of the respective
long exact sequence for $\Tor_i^A(-,B)$ looks as follows:
\[
0 \lar \Tor_0^A(B,B) \lar B\Ptens B \lar \Tor_0^A(B\ptens{A}\Omega^1 A,B)
\lar \Tor_1^A(B,B) \lar 0.
\]
We have the following commutative diagram:
\begin{equation}
\label{1psd}
\xymatrix@R-10pt@C-5pt{
0 & \Tor_0^A(B,B) \ar[l] \ar[dd]_{\bar\mu_{B,A}} & B\Ptens B \ar[l] \ar@{=}[dd]
& \Tor_0^A(B\ptens{A}\Omega^1 A,B) \ar[l] \ar[d]_{\alpha}
& \Tor_1^A(B,B) \ar[l] & 0 \ar[l]\\
&&& B\ptens{A}\Omega^1 A \ptens{A} B \ar[d]_{\check\varphi}\\
0 & B \ar[l] & B\Ptens B \ar[l]_{\mu_B} & \Omega^1 B \ar[l]_{j_B} & 0 \ar[l]
}
\end{equation}

$\mathrm{(i)}\Longrightarrow\mathrm{(iii)}$.
If $\varphi$ is a $1$-pseudoflat epimorphism, then $\Tor_1^A(B,B)=0$, and
$\bar\mu_{B,A}$ is bijective by Lemma~\ref{lemma:L1}.
Since both rows in \eqref{1psd} are exact, we see that $\check\varphi\circ\alpha$ is bijective.
Hence $\alpha$ is injective, which happens if and only if $\alpha$ is bijective, or, equivalently,
if and only if $(*)$ holds (see \eqref{Tor0-ptens}). Using again the bijectivity
of $\check\varphi\circ\alpha$, we conclude that $\check\varphi$ is bijective.

$\mathrm{(iii)}\Longrightarrow\mathrm{(i)}$.
Taking into account \eqref{Tor0-ptens}, we see that $(*)$ means precisely that $\alpha$ is bijective.
Hence $\check\varphi\circ\alpha$ is also bijective.
In view of Theorem~\ref{derpr0pf}, (iii) implies that $\varphi$ is a $0$-pseudoflat epimorphism.
Hence $\bar\mu_{B,A}$ is bijective by Lemma~\ref{lemma:L1}.
Since both lines in \eqref{1psd} are exact, and since the vertical arrows in \eqref{1psd}
are bijective, we conclude that $\Tor_1^A(B,B)=0$.
Thus $\varphi$ is $1$-pseudoflat.

$\mathrm{(ii)}\Longleftrightarrow\mathrm{(iii)}$.
Observe that $\check\varphi$ is an isomorphism if and only if for each $X\in B\bimod B$
the map $\check\varphi_X$ in \eqref{derprdi} is bijective, which is equivalent to the bijectivity of
$\widetilde\varphi_X$.
\end{proof}

\begin{remark}
Weak and strong homological epimorphisms can be nicely interpreted
in the language of derived categories (cf. \cite{GL,Meyer,Pir_qfree,BBK}).
Although we do not need this below, we find it relevant to give at least one
of such interpretations (for the convenience of those readers who are used
to think in terms of derived categories). If $A$ is a Fr\'echet algebra, then
there are two ways of making $A\lmod$ into an exact category (in Quillen's sense \cite{Quillen}).
The first (traditional) exact structure is as follows. Suppose that $M\xra{i} N\xra{p} P$
is an exact pair of morphisms in $A\lmod$ (i.e., $i$ is a kernel of $p$ and $p$ is a cokernel of $i$).
We say that such a pair is admissible if it splits in the category of Fr\'echet spaces.
It is easy to show that the collection of all admissible
exact pairs makes $A\lmod$ into an exact category.
We use the same notation $A\lmod$ to denote the resulting exact category
(this will not lead to a confusion).
Alternatively, we can make $A\lmod$ into an exact category by declaring that {\em all}
exact pairs are admissible. The fact that the collection of all exact pairs in $A\lmod$
indeed satisfies the axioms of an exact category follows from the observation
that $A\lmod$ is quasi-abelian, cf. \cite{Prosm_der_FA}.
The resulting exact category will be denoted by $A\lbarmod$.
We also let $\Fr=\CC\lmod$ and $\barFr=\CC\lbarmod$ denote the respective categories
of Fr\'echet spaces.

Homological algebra in the exact category $A\lmod$ is precisely the ``topological homology''
introduced by A.~Ya.~Helemskii \cite{X70} (see also \cite{X1,X2,EschmPut}).
The main advantage of $A\lmod$
over $A\lbarmod$ is that $A\lmod$ has enough projectives, which is not the case for $A\lbarmod$.
In fact, by a result of V.~A.~Geiler \cite{Geiler}, even the category $\barFr$
of Fr\'echet spaces does not have enough projectives. This is one of the main reasons why
homological algebra
in $A\lmod$ is developed much better than homological algebra in $A\lbarmod$.
Nevertheless, $A\lbarmod$ turns out to be useful in J.~L.~Taylor's homological approach
to multivariable spectral theory (cf. \cite{T2,EschmPut}).

Since $A\lmod$ has enough projectives, the functor $M\ptens{A}(-)\colon A\lmod\to\Fr$
is left derivable. The left derived functor of $M\ptens{A}(-)$ is denoted by
$M\Lptens{A}(-)\colon \bD^-(A\lmod)\to\bD^-(\Fr)$.
Exactly as in the algebraic case, $\Lptens{A}$ extends to a bifunctor from
$\bD^-(\rmod A)\times \bD^-(A\lmod)$ to $\bD^-(\Fr)$.
Now it is easy to see that a Fr\'echet algebra homomorphism
$\varphi\colon A\to B$ is a weak (respectively, strong) homological epimorphism
if and only if the canonical map $B\Lptens{A} B\to B$ is an isomorphism in $\bD^-(\barFr)$
(respectively, in $\bD^-(\Fr)$). This may be compared with condition (ii) of
Proposition~\ref{charepi}, which characterizes epimorphisms of Fr\'echet algebras.
\end{remark}

\section{Stein algebras}
\label{sect:Stein}

For the reader's convenience, let us recall some standard notation
(see, e.g., \cite[Chap.~0, \S4, no.~3]{EGA}). Given a morphism
$f\colon (X,\cO_X)\to (Y,\cO_Y)$ of $\CC$-ringed spaces and an $\cO_Y$-module $\cG$,
we let $f^{-1}\cG$ denote the sheaf on $X$ associated to the presheaf
$U\mapsto\cG(f(U))=\varinjlim\{\cG(V) : f(U)\subset V\}$.
We also let $f^*\cG=\cO_X\tens{f^{-1}\cO_Y}f^{-1}\cG$
denote the {\em inverse image} of $\cG$. Recall that, for each $x\in X$, we
have an $\cO_{X,x}$-module isomorphism
$(f^*\cG)_x\cong \cO_{X,x}\tens{\cO_{Y,f(x)}}\cG_{f(x)}$.

Throughout, all Stein spaces are assumed to be finite-dimensional. Let
$(X,\cO_X)$ be a Stein space. By Cartan's Theorem B, the functor $\Gamma$ of global sections
acting from the category of coherent $\cO_X$-modules to the category of all $\cO(X)$-modules
is exact. On the other hand, Cartan's Theorem A easily implies that $\Gamma$ is faithful.
Together with \cite[3.2]{Freyd}, this yields the following well-known result.

\begin{lemma}
\label{lemma:faith_ex}
A sequence $\cF\to\cG\to\cH$ of coherent $\cO_X$-modules is exact if and only if
the sequence of global sections $\cF(X)\to\cG(X)\to\cH(X)$ is exact.
\end{lemma}

Recall also (see, e.g., \cite[V.6]{GR_Stein}) that, for each coherent $\cO_X$-module $\cF$,
the space $\cF(X)$ of global sections has a canonical topology making it into a Fr\'echet space.
Moreover, $\cO(X)$ is a Fr\'echet algebra, and $\cF(X)$ is a Fr\'echet $\cO(X)$-module.
Thus $\Gamma$ can be viewed as a functor from the category of coherent $\cO_X$-modules
to the category of Fr\'echet $\cO(X)$-modules.

Given $p\in X$, we denote by $\CC_p$ the one-dimensional
$\cO(X)$-module corresponding to the evaluation map $\cO(X)\to\CC$, $a\mapsto a(p)$.

\begin{theorem}\label{holeqcond}
Let $(X,\cO_X)$ and $(Y,\cO_Y)$ be Stein spaces, let $f\colon Y\to X$
be a holomorphic map, and let $f^\bullet\colon\cO(X)\to\cO(Y)$ denote the homomorphism
induced by $f$. Then the following conditions are equivalent:
\begin{mycompactenum}
\item
$f^\bullet$ is a weak homological epimorphism;
\item
$f^\bullet$ is a $1$-pseudoflat epimorphism;
\item
$f^\bullet$ is an epimorphism, and for each $q\in Y$ we have
$\cO(Y)\perp^1_{\cO(X)}\CC_{q} $;
\item
$f$ is an open embedding.
\end{mycompactenum}
\end{theorem}
\begin{proof}
$\mathrm{(i)}\Longrightarrow\mathrm{(ii)}$: this is immediate from
Definition~\ref{def:homepi}.

$\mathrm{(ii)}\Longrightarrow\mathrm{(iii)}$. Since $\cO(X)$ and $\cO(Y)$ are
nuclear, we can apply Proposition~\ref{psflnuc}.

$\mathrm{(iii)}\Longrightarrow\mathrm{(iv)}$. We first observe that $f$ is
injective. Indeed, since $f^\bullet$ is an epimorphism, we see that the map
$\Hom(\cO(Y),\CC)\to\Hom(\cO(X),\CC)$ induced by $f^\bullet$ is
injective. By \cite[Satz 1]{For} (see also \cite[V.7.3]{GR_Stein}), for each
Stein space $Z$ we have a natural bijection $Z\cong\Hom(\cO(Z),\CC)$
taking each $z\in Z$ to the evaluation map at $z$. Therefore $f$ is
injective.

Given $q\in Y$, let $p=f(q)$, and define the ideal sheaf
$\cI\subset\cO_X$ by
\[
\cI_x=\begin{cases}
\fm_{X,x} & \text{if } x=p,\\
\cO_{X,x} & \text{if } x\ne p,
\end{cases}
\]
where $\fm_{X,x}$ is the maximal ideal of $\cO_{X,x}$.
By \cite[Satz 6.4]{For}, there exists a resolution
\begin{equation}
\label{cP}
0 \lar \cO_X/\cI \lar \cO_X \lar \cP_1 \lar \cP_2\lar \cdots ,
\end{equation}
where all the $\cP_i$'s are free $\cO_X$-modules of finite rank, and where
$\cO_X\to\cO_X/\cI$ is the quotient map. Taking the sections over $X$
and applying Cartan's Theorem B, we obtain an exact complex
\begin{equation}
\label{cPmod}
 0 \lar \CC_p \lar \cO(X) \lar P_1\lar P_2\lar\cdots
\end{equation}
of Fr\'echet $\cO(X)$-modules. Note that $\CC_p=\CC_q$ in $\cO(X)\lmod$.

By Proposition~\ref{prop:nucl_Tor}, we can use \eqref{cPmod}
to calculate $\Tor_i^{\cO(X)}(\cO(Y),\CC_q)$.
Condition (iii) implies that
\begin{align*}
\Tor_1^{\cO(X)}(\cO(Y),\CC_q)&=0,\quad\text{and}\\
\Tor_0^{\cO(X)}(\cO(Y),\CC_q)&\cong
\cO(Y)\ptens{\cO(X)}\CC_q\cong\cO(Y)\ptens{\cO(Y)}\CC_q\cong\CC_q
\end{align*}
canonically (see Proposition~\ref{charepi}).
Hence we have an exact sequence
\begin{equation}
\label{P_Y}
0 \lar \CC_q \lar \cO(Y)\ptens{\cO(X)} \cO(X) \lar \cO(Y)\ptens{\cO(X)} P_1
\lar\cO(Y)\ptens{\cO(X)} P_2
\end{equation}
of Fr\'echet $\cO(Y)$-modules.

Now observe that the functors $\cF\mapsto\cO(Y)\ptens{\cO(X)} \cF(X) $ and
$\cF\mapsto (f^*\cF)(Y)$ obviously agree on the category of free
$\cO_X$-modules of finite rank. Hence~\eqref{P_Y} is isomorphic to the
sequence obtained by applying $\Gamma(Y,-)$ to
\begin{equation}
\label{f*cP}
0 \lar \cO_Y/\cI' \lar \cO_Y\cong f^*\cO_X \lar f^*\cP_1 \lar f^*\cP_2,
\end{equation}
where $\cI'\subset\cO_Y$ is the ideal sheaf given by
\[
\cI'_y=\begin{cases}
\fm_{Y,y} & \text{if } y=q,\\
\cO_{Y,y} & \text{if } y\ne q.
\end{cases}
\]
Applying Lemma~\ref{lemma:faith_ex}, we conclude that \eqref{f*cP} is exact.

Consider now the stalks of \eqref{cP} over $p$ and the stalks of
\eqref{f*cP} over $q$. For notational convenience, let $A=\cO_{X,p}$,
$B=\cO_{Y,q}$, $\fm_A=\fm_{X,p}$, $\fm_B=\fm_{Y,q}$, and
$F_i=(\cP_i)_x$. Let also $\varphi\colon A\to B$ denote the homomorphism induced
by $f$. We have two exact sequences
\begin{align}
\label{FA}
0 &\lar A/\fm_A \lar A \lar F_1 \lar F_2 \lar \cdots, \\
\label{FB}
0 &\lar B/\fm_B \lar B\tens{A} A \lar B\tens{A} F_1 \lar B\tens{A} F_2.
\end{align}
Comparing \eqref{FA} with \eqref{FB}, we see that
\begin{equation}
\label{Tor=0}
\Tor_1^A(B,A/\fm_A)=0,
\end{equation}
where $\Tor_1^A$ stands for the purely algebraic $\Tor$-functor.
Also, the exactness of \eqref{FA} and \eqref{FB} implies
that $B\tens{A} (A/\fm_A)\cong B/\fm_B$ via the map $b\otimes
(a+\fm_A)\mapsto b\varphi(a)+\fm_B$. It is readily verified that the latter
condition is equivalent to the equality $B\varphi(\fm_A)=\fm_B$. By
\cite[2.2.3]{GR_anloc}, this means that $\varphi$ is onto. In particular,
$B$ is a finitely generated $A$-module.  Since $A$ is Noetherian
\cite[2.0.1]{GR_anloc}, $B$ is a finitely presented $A$-module. Combining
this with~\eqref{Tor=0} and applying \cite[Chap. II, \S3, no.~2]{Bour_AC},
we see that $B$ is free over $A$. Since $\dim (B/B\fm_A)=\dim (B/\fm_B)=1$,
it follows from \cite[Appendix, 2.7 (i)]{GR_anloc}
that $\varphi$ is an isomorphism. By \cite[0.23]{Fischer}, this means
exactly that $f$ is locally biholomorphic. Since $f$ is also injective (see
above), we conclude that $f$ is an open embedding.

$\mathrm{(iv)}\Longrightarrow\mathrm{(i)}$.
Without loss of generality, we
may assume that $Y$ is a Stein open subset of $X$ and that
$\cO_Y=\cO_X|\,Y$. Thus $f^\bullet\colon\cO(X)\to\cO(Y)$ is the restriction map.
Recall from \cite[Corollary 4.2.5]{EschmPut} that, for each morphism
$g\colon (Z,\cO_Z)\to (X,\cO_X)$ of Stein spaces,
and for each coherent $\cO_Z$-module $\cF$, we have
$\cO(Y)\perp_{\cO(X)}\cF(Z)$, and $\cO(Y)\ptens{\cO(X)}\cF(Z)\cong\cF(g^{-1}(Y))$.
Letting $Z=Y$, $\cF=\cO_Y$, and $g=(\text{inclusion } Y\hookrightarrow X)$,
we obtain $\cO(Y)\perp_{\cO(X)}\cO(Y)$ and $\cO(Y)\ptens{\cO(X)}\cO(Y)\cong\cO(Y)$.
This means exactly that the restriction map $\cO(X)\to\cO(Y)$ is a weak
homological epimorphism.
\end{proof}

\begin{remark}
We have already pointed out in Section~\ref{sect:intro} that, if $f\colon Y\to X$ is an
open embedding of Stein spaces, then $f^\bullet\colon\cO(X)\to\cO(Y)$ is not necessarily
flat. Thus the class of $1$-pseudoflat Fr\'echet algebra
epimorphisms is essentially larger than the class
of flat epimorphisms, even in the commutative case.
It is interesting to compare this with recent purely algebraic results from \cite{AH,Ba_Posic}.
Namely, a $1$-pseudoflat epimorphism $A\to B$ of commutative rings
is necessarily flat provided that either (a) $A$ is Noetherian \cite[Prop. 4.5]{AH}, or
(b) the projective dimension of $B$ over $A$ is $\le 1$ \cite[Remark 16.9]{Ba_Posic}.
While property (a) rarely holds in the functional analytic context
(for example, the algebras of holomorphic functions on Stein manifolds
are never Noetherian), property (b) is more common. For example, if $\DD\subset\CC$
is the open unit disc, then the restriction
map $\cO(\CC)\to\cO(\DD)$ is a $1$-pseudoflat epimorphism by Theorem~\ref{holeqcond}
(actually, by \cite[Prop. 3.1]{T2}), satisfies (b) by \cite[Theorem V.1.8]{X1}, but is not flat
by \cite{Pir_Stein}. This shows that the above-mentioned result
of \cite{Ba_Posic} has no analog in the Fr\'echet algebra setting.
\end{remark}

\section{Algebras of $C^\infty$-functions}
\label{sect:C-inf}

In this section, we prove a $C^\infty$-analog of Theorem~\ref{holeqcond}.
Towards this goal, we need two lemmas. Let $A$, $B$, $C$ be Fr\'echet algebras,
$N$ be a Banach $B$-$C$-bimodule, and $P$ be a Fr\'echet $A$-$C$-bimodule.
Then $\h_C(N,P)$ is a Fr\'echet space under the topology of uniform convergence
on the unit ball of $N$. Moreover, $\h_C(N,P)$ is a Fr\'echet $A$-$B$-bimodule with
respect to the actions
\[
(a\cdot\varphi)(x)=a\cdot\varphi(x),
\quad (\varphi\cdot b)(x)=\varphi(b\cdot x)
\quad (a\in A,\; b\in B,\; x\in N,\; \varphi\in\h_C(N,P)).
\]

\begin{lemma}
\label{lemma:adj_ass}
Let $A$, $B$, $C$ be Fr\'echet algebras, $M$ be a Fr\'echet $A$-$B$-bimodule,
$N$ be a Banach $B$-$C$-bimodule, and $P$ be a Fr\'echet $A$-$C$-bimodule.
Then there exists a vector space isomorphism
\[
{_A}\h_C(M\ptens{B}N,P)\lriso {_A}\h_B(M,\h_C(N,P)),
\qquad \varphi\mapsto (x\mapsto (y\mapsto\varphi(x\otimes y))).
\]
\end{lemma}

We omit the standard proof (cf. \cite[II.5.22]{X1}, \cite[Prop. 3.2]{Pir_msb}).

\begin{lemma}
\label{lemma:Der_bij}
Let $\varphi\colon A\to B$ be a Fr\'echet algebra epimorphism,
and let $\epsilon\colon B\to\CC$
be a continuous homomorphism. Let $\CC_\epsilon$ denote the one-dimensional $B$-bimodule
corresponding to $\epsilon$. Suppose that $B\perp^1_A\CC_\epsilon$.
Then $\wt{\varphi}_{\CC_\epsilon}\colon\Der(B,\CC_\epsilon)\to\Der(A,\CC_\epsilon)$
is a bijection.
\end{lemma}
\begin{proof}
As in the proof of Theorem~\ref{kalch}, the admissible sequence
\eqref{B_tens_Omega} yields a long exact sequence of
$\Tor_i^A(-,\CC_\epsilon)$, whose low-dimensional segment fits into
the following commutative diagram:
\begin{equation}
\label{BCperp}
\xymatrix@R-10pt@C-5pt{
0 & \Tor_0^A(B,\CC_\epsilon) \ar[l] \ar[d]_{\alpha}
& B\Ptens \CC_\epsilon \ar[l] \ar@{=}[dd]
& \Tor_0^A(B\ptens{A}\Omega^1 A,\CC_\epsilon) \ar[l] \ar[d]_{\alpha'}
& \Tor_1^A(B,\CC_\epsilon) \ar[l] & 0 \ar[l]\\
& B\ptens{A}\CC_\epsilon \ar[d]_{\beta}
&& B\ptens{A}\Omega^1 A \ptens{A} \CC_\epsilon \ar[d]_{\gamma}\\
0 & \CC_\epsilon \ar[l] & B\Ptens \CC_\epsilon \ar[l]
& \Omega^1 B\ptens{B}\CC_\epsilon \ar[l] & 0 \ar[l]
}
\end{equation}
Here $\alpha$ and $\alpha'$ are the canonical maps from
$\Tor_0^A(-,-)$ to $(-)\ptens{A}(-)$, $\beta$ is the canonical map from
$B\ptens{A}\CC_\epsilon$ to $B\ptens{B}\CC_\epsilon\cong\CC_\epsilon$,
and $\gamma$ corresponds to $\check\varphi\tens{B}\id_{\CC_\epsilon}$
under the identification
$(B\ptens{A}\Omega^1 A\ptens{A} B)\ptens{B}\CC_\varepsilon\cong
B\ptens{A}\Omega^1 A \ptens{A} \CC_\epsilon$.
Finally, the bottom row of \eqref{BCperp} is obtained from \eqref{can_B}
via $(-)\ptens{B}\CC_\epsilon$, so it is exact because \eqref{can_B} splits in $\rmod B$.

Since $\varphi$ is an epimorphism, Proposition~\ref{charepi} implies that
$\beta$ is bijective. Since $B\perp^1_A\CC_\epsilon$, we see that
$\alpha$ is bijective and that $\Tor_1^A(B,\CC_\epsilon)=0$.
Together with the fact that both lines in \eqref{BCperp}
are exact, this implies that $\gamma\circ\alpha'$ is bijective.
Hence $\alpha'$ is injective, or, equivalently, bijective
(see \eqref{Tor0-ptens}), which in turn implies that $\gamma$ is bijective.
Thus $\gamma$ is an isomorphism in $B\lmod$.

Applying ${_B}\h(-,\CC_\epsilon)$ to $\gamma$, we obtain a vector space
isomorphism
\begin{equation}
\label{gamma-iso}
{_B}\h(\Omega^1 B\ptens{B}\CC_\epsilon,\CC_\epsilon)
\lriso {_B}\h(B\ptens{A}\Omega^1 A \ptens{A} \CC_\epsilon,\CC_\epsilon).
\end{equation}
Observe that there is a $B$-bimodule isomorphism
$\h_{\CC}(\CC_\epsilon,\CC_\epsilon)\cong\CC_\epsilon$
given by $f\mapsto f(1)$. Together with Lemma~\ref{lemma:adj_ass},
this implies that
\begin{equation}
\label{chainiso1}
{_B}\h(\Omega^1 B\ptens{B}\CC_\epsilon,\CC_\epsilon)
\cong {_B}\h_B(\Omega^1 B,\h_{\CC}(\CC_\epsilon,\CC_\epsilon))
\cong {_B}\h_B(\Omega^1 B,\CC_\epsilon),
\end{equation}
and
\begin{equation}
\label{chainiso2}
\begin{split}
{_B}\h(B\ptens{A}\Omega^1 A \ptens{A} \CC_\epsilon,\CC_\epsilon)
&\cong {_B}\h((B\ptens{A}\Omega^1 A\ptens{A} B)\ptens{B}\CC_\varepsilon,\CC_\epsilon)\\
&\cong {_B}\h_B(B\ptens{A}\Omega^1 A\ptens{A} B,\h_{\CC}(\CC_\epsilon,\CC_\epsilon))\\
&\cong {_B}\h_B(B\ptens{A}\Omega^1 A\ptens{A} B,\CC_\epsilon).
\end{split}
\end{equation}
Under the identifications \eqref{chainiso1} and \eqref{chainiso2},
the isomorphism \eqref{gamma-iso} becomes
\begin{equation}
\label{gamma-iso2}
{_B}\h_B(\Omega^1 B,\CC_\epsilon)
\lriso {_B}\h_B(B\ptens{A}\Omega^1 A \ptens{A} B,\CC_\epsilon).
\end{equation}
A routine calculation shows that \eqref{gamma-iso2} is nothing but the map
$\check\varphi_{\CC_\epsilon}$ from diagram \eqref{derprdi}
(in which we let $X=\CC_\epsilon$).
This readily implies that
$\wt{\varphi}_{\CC_\epsilon}\colon\Der(B,\CC_\epsilon)\to\Der(A,\CC_\epsilon)$
is a  vector space isomorphism.
\end{proof}

Let $X$ be a $C^\infty$-manifold.
We denote by $C^\infty(X)$ the Fr\'{e}chet algebra of infinitely differentiable
$\CC$-valued functions on $X$.
Similarly to the holomorphic case (see Section~\ref{sect:Stein}), given $p\in X$, we denote by $\CC_p$
the one-dimensional $C^\infty(X)$-module
corresponding to the evaluation map $C^\infty(X)\to\CC$, $a\mapsto a(p)$.

\begin{theorem}\label{BIGTHsmman}
Let $X$ and $Y$ be $C^\infty$-manifolds, let $f\colon Y\to X$
be a smooth map, and let $f^\bullet\colon C^\infty(X)\to C^\infty(Y)$ denote the
homomorphism induced by $f$. Then the following conditions are equivalent:
\begin{mycompactenum}
\item
$f^\bullet$ is a projective epimorphism, i.e., $f^\bullet$ is an epimorphism and
$C^\infty(Y)$ is projective in $C^\infty(X)\lmod$;
\item
$f^\bullet$ is a flat epimorphism, i.e., $f^\bullet$ is an epimorphism and
$C^\infty(Y)$ is flat in $C^\infty(X)\lmod$;
\item
$f^\bullet$ is a strong homological epimorphism;
\item
$f^\bullet$ is a weak homological epimorphism;
\item
$f^\bullet$ is a $1$-pseudoflat epimorphism;
\item
$f^\bullet$ is an epimorphism, and for each $q\in Y$ we have
$C^\infty(Y)\perp^1_{C^\infty(X)}\CC_{q}$;
\item
$f$ is an open embedding.
\end{mycompactenum}
\end{theorem}
\begin{proof}
$\mathrm{(i)}\Longrightarrow\mathrm{(ii)}\Longrightarrow\mathrm{(iv)}$,
$\mathrm{(i)}\Longrightarrow\mathrm{(iii)}\Longrightarrow\mathrm{(iv)}$,
$\mathrm{(iv)}\Longrightarrow\mathrm{(v)}$: this is trivial.

$\mathrm{(v)}\Longrightarrow\mathrm{(vi)}$.
Since $C^\infty(X)$ and $C^\infty(Y)$ are
nuclear, we can apply Proposition~\ref{psflnuc}.

$\mathrm{(vi)}\Longrightarrow\mathrm{(vii)}$.
As in the proof of Theorem~\ref{holeqcond},
we first observe that $f$ is
injective. Indeed, since $f^\bullet$ is an epimorphism, we see that the map
$\Hom(C^\infty(Y),\CC)\to\Hom(C^\infty(X),\CC)$ induced by $f^\bullet$ is
injective. By \cite[Theorem 7.2]{Nestruev}, for each
smooth manifold $Z$ we have a natural bijection $Z\cong\Hom(C^\infty(Z),\CC)$
taking each $z\in Z$ to the evaluation map at $z$. Therefore $f$ is
injective.

To complete the argument we need to show that $f$ is a local diffeomorphism.
By the Inverse Function Theorem, it suffices to check that for each $q\in Y$ the
tangent map $df_q\colon T_q(Y)\to T_{f(q)}(X)$ is a vector space isomorphism.
Identifying $T_q(Y)$ with $\Der(C^\infty(Y),\CC_q)$ (see, e.g., \cite[III.3.1]{CCC1}),
we see that $df_q$ is nothing but
\[
(\wt{f^\bullet})_{\CC_q}\colon\Der(C^\infty(Y),\CC_q)\to\Der(C^\infty(X),\CC_q)
\]
(as in Theorem~\ref{holeqcond}, we identify $\CC_q$ with $\CC_{f(q)}$ in $C^\infty(X)\lmod$.)
By Lemma~\ref{lemma:Der_bij}, $(\wt{f^\bullet})_{\CC_q}$ is a bijection.
Hence $f$ is an open embedding.

$\mathrm{(vii)}\Longrightarrow\mathrm{(i)}$.
Without loss of generality, we may assume that $Y$ is an open subset of $X$.
Thus $f^\bullet\colon C^\infty(X)\to C^\infty(Y)$ is the restriction map.
A standard argument involving bump functions shows that the image
of $f^\bullet$ is dense in $C^\infty(Y)$. Hence $f^\bullet$ is an epimorphism.
By \cite[Theorem 2]{Og}, $C^\infty(Y)$ is projective\footnote[1]{In \cite{Og},
the projectivity of $C^\infty(Y)$ over $C^\infty(X)$ is proved under the assumption
that $Y$ is contained in a coordinate neighborhood. However, the proof readily carries over
to the general case.} over $C^\infty(X)$. This completes the proof.
\end{proof}

\section{Concluding remarks and questions}
\label{sect:rem}

An obvious difference between our main results, i.e., Theorems~\ref{holeqcond}
and \ref{BIGTHsmman}, is that the strong conditions (i)--(iii) of Theorem~\ref{BIGTHsmman}
are missing in Theorem~\ref{holeqcond}.
We have already mentioned in Section~\ref{sect:intro} that open embeddings
of Stein spaces usually do not satisfy condition (ii) (and, {\em a fortiori}, do not satisfy condition (i))
of Theorem~\ref{BIGTHsmman}. However, the situation with condition (iii) is not that clear.
In fact, we do not know the answer to the following question.

\begin{question}
\label{quest:strong}
Let $(X,\cO_X)$ be a Stein space, and let $(Y,\cO_Y)$ be a Stein open subspace of $(X,\cO_X)$.
Is the restriction map $\cO(X)\to\cO(Y)$ a strong homological epimorphism?
\end{question}

In the special case where $X=\CC^n$ and $Y$ is a polydomain (i.e., a product of one-dimensional
open subsets of $\CC$), the answer to Question~\ref{quest:strong} is positive
by \cite[Prop. 4.3]{T2}. On the other hand, the answer seems to be unknown already in the case
where $X=\CC^2$ and $Y$ is the open unit ball.

Another difference between Theorems~\ref{holeqcond} and \ref{BIGTHsmman}
is in the ``degree of singularity'' of the objects considered therein.
Indeed, Stein spaces are not necessarily reduced (i.e., their structure sheaves are
allowed to have nilpotents), and even reduced Stein spaces are not necessarily smooth
(i.e., are not necessarily locally isomorphic to an open subset of $\CC^n$).
On the other hand, $C^\infty$-manifolds are reduced and smooth (in the
appropriate sense) by definition.
The theory of $C^\infty$-differentiable spaces \cite{NaSa} studies geometric objects
which are more general than $C^\infty$-manifolds, and
which can be viewed as ``correct'' $C^\infty$-analogs of Stein spaces.
In particular, $C^\infty$-differentiable spaces may have singular points and may be non-reduced.
(Note that \cite{NaSa} deals with $\R$-valued functions only, but an extension
to $\CC$-valued functions is straightforward.)
It would be interesting to characterize
open embeddings of $C^\infty$-differentiable
spaces (at least in the affine case) in the spirit of Theorem~\ref{BIGTHsmman}.

In its full form, Theorem~\ref{BIGTHsmman} does not extend to
$C^\infty$-differentiable spaces. For example, consider the map
$\pi\colon C^\infty(\R^n)\to \CC[[x_1,\ldots,x_n]]$ which takes each
smooth function on $\R^n$ to its Taylor series at $0$.
Using the Koszul
resolution, one can prove that $\pi$ is a strong homological epimorphism
(cf. \cite[Prop. 4.4]{T2}). At the same time,
the corresponding map $\pi^*$ of affine $C^\infty$-differentiable spaces is
not an open embedding. On the other hand, applying $(-)\ptens{A}
\CC[[x_1,\ldots,x_n]]$ to the inclusion $I\hookrightarrow A$, where $A=C^\infty(\R^n)$
and $I=\{ f\in A : f(0)=0\}$, one can easily show that $\pi$ is
not flat.

The $C^\infty$-differentiable space with one-point spectrum that
corresponds to $\CC[[x_1,\ldots,x_n]]$ is a special case of
$\mathbf{W}_{Y/X}$, the so-called  Whitney subspace  of $Y$, where
$Y$ is a closed subset of a $C^\infty$-differentiable space $X$
\cite[Corollary 5.10]{NaSa}. In the case where $X$ is an open
subset of $\R^n$, the map $\mathbf{W}_{Y/X}\to X$ corresponds to
the quotient homomorphism $C^\infty(X)\to C^\infty(X)/W_{Y/X}$,
where $W_{Y/X}$ is the ideal of functions whose derivatives
of all orders vanish on $Y$. Normally, $\mathbf{W}_{Y/X}\to X$ is not
an open embedding (in fact, it is always a closed embedding).
Nevertheless, we have the following result.

\begin{prop}
Let $X$ be an open subset of $\R^n$, and let $Y$ be a closed subset of
$X$. Then the quotient map $C^\infty(X)\to C^\infty(X)/W_{Y/X}$ is
a $1$-pseudoflat epimorphism.
\end{prop}
\begin{proof}
It follows from \cite[Lemme 2.4]{Tou} that any real-valued
function $f\in W_{Y/X}$ has the form $f=f_1f_2$, where $f_1$
and $f_2$ are again in $W_{Y/X}$. Since $C^\infty(X)$ is nuclear,
we can apply Proposition~\ref{prop:quot_1psd}.
\end{proof}

The above remarks lead naturally to the following
two questions on
$C^\infty$-differentiable spaces.

\begin{question}
Can open embeddings of affine $C^\infty$-differentiable spaces be
characterized in terms of projectivity or flatness, as in
Theorem~{\upshape\ref{BIGTHsmman}}?
\end{question}

\begin{question}
Let $(X,\cO_X)$ be an affine $C^\infty$-differentiable
space, and let $Y$ be a closed subset of $X$.
Is the quotient map $\cO(X)\to
\cO(\mathbf{W}_{Y/X})$ a $1$-pseudoflat epimorphism?
Is it a weak homological epimorphism?
Is it a strong homological epimorphism?
\end{question}

\begin{ackn}
The authors thank the referee for useful comments that significantly
improved the presentation.
\end{ackn}

\end{document}